\documentclass[12pt]{amsart}

\usepackage{amsmath}
\usepackage{amscd}
\usepackage{amssymb}
\usepackage{amsfonts}
\usepackage{amsthm}
\usepackage{enumerate}
\usepackage{hyperref}
\usepackage{color}
\usepackage{graphicx}

\theoremstyle{plain}
\newtheorem{thm}{Theorem}[section]

\newtheorem{prop}[thm]{Proposition}
\newtheorem{lem}[thm]{Lemma}
\newtheorem{cor}[thm]{Corollary}

\newtheorem{conj}[thm]{Conjecture}

\theoremstyle{definition}
\newtheorem{dfn}[thm]{Definition}

\newtheorem{exmps}[thm]{Examples}
\newtheorem{rem}[thm]{Remark}

\newtheorem{dfns-rems}[thm]{Definitions and Remarks}
\newtheorem{notas-rems}[thm]{Notations and Remarks}
\newtheorem{exmps-rems}[thm]{Examples and Remarks}

\DeclareMathOperator{\ind-match}{ind-match}

\makeatletter
\@namedef{subjclassname@2020}{%
  \textup{2020} Mathematics Subject Classification}
\makeatother

\DeclareMathOperator{\ord-match}{ord-match}
\DeclareMathOperator{\ind-match}{ind-match}

\begin{document}


\title[Regularity of squarefree part of symbolic powers]{On the Regularity of squarefree part of symbolic powers of edge ideals}


\author[S. A. Seyed Fakhari]{S. A. Seyed Fakhari}

\address{S. A. Seyed Fakhari, Departamento de Matem\'aticas\\Universidad de los Andes\\Bogot\'a\\Colombia.}

\email{s.seyedfakhari@uniandes.edu.co}

\begin{abstract}
Assume that $G$ is a graph with edge ideal $I(G)$. For every integer $s\geq 1$, we denote the squarefree part of the $s$-th symbolic power of $I(G)$ by $I(G)^{\{s\}}$. We determine an upper bound for the regularity of $I(G)^{\{s\}}$ when $G$ is a chordal graph. If $G$ is a Cameron-Walker graphs, we compute ${\rm reg}(I(G)^{\{s\}}$ in terms of the induced matching number of $G$. Moreover, for any graph $G$, we provide sharp upper bounds for ${\rm reg}(I(G)^{\{2\}})$ and ${\rm reg}(I(G)^{\{3\}})$.
\end{abstract}






\maketitle


\section{Introduction} \label{sec1}

Let $\mathbb{K}$ be a field and $S = \mathbb{K}[x_1,\ldots,x_n]$  be the
polynomial ring in $n$ variables over $\mathbb{K}$. Suppose that $M$ is a graded $S$-module with minimal free resolution
$$0  \longrightarrow \cdots \longrightarrow  \bigoplus_{j}S(-j)^{\beta_{1,j}(M)} \longrightarrow \bigoplus_{j}S(-j)^{\beta_{0,j}(M)}   \longrightarrow  M \longrightarrow 0.$$
The integer $\beta_{i,j}(M)$ is called the $(i,j)$th graded Betti number of $M$. The Castelnuovo--Mumford regularity (or simply, regularity) of $M$,
denoted by ${\rm reg}(M)$, is defined as
$${\rm reg}(M)=\max\{j-i|\ \beta_{i,j}(M)\neq0\},$$
and it is an important invariant in commutative algebra and algebraic geometry.

There is a natural correspondence between quadratic squarefree monomial ideals of $S$ and finite simple graphs with $n$ vertices. To every simple graph $G$ with vertex set $V(G)=\big\{x_1, \ldots, x_n\big\}$ and edge set $E(G)$, we associate its {\it edge ideal} $I=I(G)$ defined by
$$I(G)=\big(x_ix_j: x_ix_j\in E(G)\big)\subseteq S.$$

Computing and finding bounds for the regularity of edge ideals and their powers have been studied by a number of researchers (see for example \cite{b}, \cite{bbh}, \cite{bht}, \cite{dhs}, \cite{ha}, \cite{hh1}, \cite{js}, \cite{s7}, \cite{sy}, and \cite{wo}).

Let $I$ be a monomial ideal. By {\it squarefree part} of $I$, we mean the monomial ideal which is generated by squarefree monomials of $I$. In \cite{ehhs}, Erey, Herzog, Hibi and Saeedi Madani initiated the study of regularity of squarefree part of powers of edge ideals. This study was continued in \cite{eh} and \cite{s11}. In this paper, we replace ordinary powers by symbolic powers and investigate the regularity of squarefree part of symbolic powers of edge ideals. For a monomial ideal $I$ and for a positive integer $s$, let $I^{\{s\}}$ denote the squarefree part of the $s$-th symbolic power of $I$. Erey et al.
\cite[Theorem 2.1]{ehhs} proved that for every graph $G$ with induced matching number $\ind-match(G)$ and for any positive integer $s\leq \ind-match(G)$,$${\rm reg}(I(G)^{[s]})\geq s+\ind-match(G),$$where $I(G)^{[s]}$ denotes the squarefree part of $I(G)^s$. In Theorem \ref{indmatch}, we will prove that the same inequality is true if one replaces $I(G)^{[s]}$ by $I(G)^{\{s\}}$.

In Section \ref{sec3}, we study the regularity of squarefree part of symbolic power of edge ideals of chordal graphs. For any squarefree monomial ideal $I$, it is not difficult to see that $I^{\{s\}}\neq 0$ if and only if $s\leq \mathfrak{ht}(I)$, where $\mathfrak{ht}(I)$ denotes the height of $I$ (see Proposition \ref{zero}). In Theorem \ref{mainchord}, we will prove that for any chordal graph $G$ and for every integer $s$ with $1\leq s\leq \mathfrak{ht}(I(G))$, we have$${\rm reg}(I(G)^{\{s\}})\leq s+\ord-match(G)$$where $\ord-match(G)$ denotes the ordered matching number of $G$ (see Definition \ref{om}). Moreover, we will see in Remark \ref{sharpchord} that the above inequality is sharp. On the other hand, we know from \cite{s10} that the regularity of $I(G)^{(s)}$ only depends on the induced matching number of $G$. Remark \ref{indchord} shows that one can not expect the same behavior for the regularity of $I(G)^{\{s\}}$.

In Section \ref{sec4}, we determine the regularity of squarefree part of symbolic powers of edge ideals of Cameron-Walker graphs. As the main result of that section, we will prove in Theorem \ref{maincw} that for any Cameron-Walker graph $G$ and for every integer $s$ with $1\leq s\leq \mathfrak{ht}(I(G))$,$${\rm reg}(I(G)^{\{s\}})=s+\ind-match(G).$$We know from \cite[Theorem 4.3]{s11} that for any Cameron-Walker graph $G$ and for every integer $s$ with $1\leq s\leq {\rm match}(G)$, we have ${\rm reg}(I(G)^{[s]})=s+\ind-match(G)$. Thus, for this class of graphs, the equality ${\rm reg}(I(G)^{\{s\}})={\rm reg}(I(G)^{[s]})$ holds, for each integer $s$ with $1\leq s\leq {\rm match}(G)$. It is natural to ask whether the same is true for any arbitrary graph $G$. However, as we will see in Remark \ref{orsynoeq}, the answer is negative.

In Sections \ref{sec5} and \ref{sec6}, we determine sharp upper bounds for the regularity of $I(G)^{\{2\}}$ and $I(G)^{\{3\}}$. More precisely, we will prove in Theorem \ref{sym2} that for any graph $G$ with $I(G)^{\{2\}}\neq 0$, we have$${\rm reg}(I(G)^{\{2\}})\leq \min\big\{{\rm reg}(I(G))+2, {\rm match}(G)+2\big\},$$where ${\rm match}(G)$ denotes the matching number of $G$. In \cite{ehhs}, Erey et al. conjectured that for any graph $G$ and for any integer $s$ with $I(G)^{[s]}\neq 0$, the inequality ${\rm reg}(I(G)^{[s]})\leq {\rm match}(G)+s$ holds. We expect the same statement to be true if one replaces $I(G)^{[s]}$ by $I(G)^{\{s\}}$. So, we propose the following conjecture.

\begin{conj} \label{conjsym}
For any graph $G$ and for any integer $s$ with $I(G)^{\{s\}}\neq 0$, we have ${\rm reg}(I(G)^{\{s\}})\leq {\rm match}(G)+s$.
\end{conj}

Note that Conjecture \ref{conjsym} is true for\\
$\bullet$ $s=1$ (\cite[Theorem 6.7]{hv}),\\
$\bullet$ $s=2$ (Theorem \ref{sym2}),\\
$\bullet$ chordal graphs (Theorem \ref{mainchord}), and\\
$\bullet$ Cameron-Walker graphs (Theorem \ref{maincw}).

Unfortunately, we are not able to verify Conjecture \ref{conjsym} for $s=3$. However, we will prove in Theorem \ref{mainpo3} that for any graph $G$ with $n$ vertices such that $I(G)^{\{3\}}\neq 0$, we have$${\rm reg}(I(G)^{\{3\}})\leq\min\Big\{\Big\lfloor\frac{n}{2}\Big\rfloor+3, {\rm reg}(I(G))+4\Big\}.$$


\section{Preliminaries and basic results} \label{sec2}

In this section, we provide the definitions and basic facts which will be used in the next sections.

All graphs in this paper are simple, i.e., have no loops and no multiple edges. Let $G$ be a graph with vertex set $V(G)=\big\{x_1, \ldots,
x_n\big\}$ and edge set $E(G)$. We identify the vertices (resp. edges) of $G$ with variables (resp. corresponding quadratic monomials) of $S$. For a vertex $x_i$, the {\it neighbor set} of $x_i$ is $N_G(x_i)=\{x_j\mid x_ix_j\in E(G)\}$. We set $N_G[x_i]=N_G(x_i)\cup \{x_i\}$. The {\it degree} of $x_i$, denoted by ${\rm deg}_G(x_i)$ is the cardinality of $N_G(x_i)$. A vertex of degree one is called a {\it leaf}. An edge $e\in E(G)$ is a {\it pendant edge}, if it is incident to a leaf. A {\it pendant triangle} of $G$ is a triangle $T$ of $G$, with the property that exactly two vertices of $T$ have degree two in $G$. A {\it star triangle} is the graph consisting of finitely many triangles sharing exactly one vertex. The graph $G$ is {\it bipartite} if there exists a partition $V(G)=A\cup B$ such
that each edge of $G$ is of the form $x_ix_j$ with $x_i\in A$ and $x_j\in B$. If moreover, every vertex of $A$ is adjacent to every vertex of $B$, then we say that $G$ is a {\it complete bipartite} graph and denote it by $K_{a,b}$, where $a=|A|$ and $b=|B|$. A subgraph $H$ of $G$ is called {\it induced} provided that two vertices of $H$ are adjacent if and only if they are adjacent in $G$. A graph $G$ is called {\it chordal} if it has no induced cycle of length at least four. A subset $W$ of $V(G)$ is a {\it clique} of $G$ if every two distinct vertices of $W$ are adjacent in $G$. A vertex $x$ of $G$ is a {\it simplicial vertex} if $N_G(x)$ is a clique. It is well-known that every chordal graph has a simplicial vertex. The {\it complementary graph} $\overline{G}$ is a graph with $V(\overline{G})=V(G)$ and $E(\overline{G})=\{x_ix_j\mid x_ix_j\notin E(G)\}$. For every subset $U\subset V(G)$, the graph $G\setminus U$ has vertex set $V(G\setminus U)=V(G)\setminus U$ and edge set $E(G\setminus U)=\{e\in E(G)\mid e\cap U=\emptyset\}$. If $U=\{x_i\}$ is a singleton, we write $G\setminus x_i$ instead of $G\setminus \{x_i\}$. A subset $A$ of $V(G)$ is called an {\it independent subset} of $G$ if there are no edges among the vertices of $A$.

Let $G$ be a graph. A subset $C$ of $V(G)$ is called a {\it vertex cover} of $G$ if every edge of $G$ is incident to at least one vertex of $C$. A vertex cover $C$ is a {\it minimal vertex cover} if no proper subset of $C$ is a vertex cover of $G$. The set of minimal vertex covers of $G$ will be denoted by $\mathcal{C}(G)$. For every subset $A$ of $\big\{x_1, \ldots, x_n\big\}$, $\mathfrak{p}_A$ denotes the monomial prime ideal which is generated by the variables belonging to $A$. It is well-known that for every graph $G$ with edge ideal $I(G)$,$$I(G)=\bigcap_{C\in \mathcal{C}(G)}\mathfrak{p}_C.$$In particular, the height of $I(G)$ is the smallest cardinality of minimal vertex covers of $G$.

Let $G$ be a graph. A subset $M\subseteq E(G)$ is a {\it matching} if $e\cap e'=\emptyset$, for every pair of edges $e, e'\in M$. The cardinality of the largest matching of $G$ is called the {\it matching number} of $G$ and is denoted by ${\rm match}(G)$. It is clear that for every positive integer $s$, the ideal $I(G)^{[s]}$ is generated by monomials of the form $e_1\ldots e_s$, where $\{e_1, \ldots, e_s\}$ is a matching of $G$. A matching $M$ of $G$ is an {\it induced matching} of $G$ if for every pair of edges $e, e'\in M$, there is no edge $f\in E(G)\setminus M$ with $f\subset e\cup e'$. The cardinality of the largest induced matching of $G$ is the {\it induced  matching number} of $G$ and is denoted by $\ind-match(G)$.

\begin{dfn} \label{om}
Let $G$ be a graph, and let $M=\{a_ib_i\mid 1\leq i\leq r\}$ be a
nonempty matching of $G$. We say that $M$ is an {\it ordered matching} of
$G$ if the following hold:
\begin{itemize}
\item[(1)] $A:=\{a_1,\ldots,a_r\} \subseteq V(G)$ is a set of
    independent vertices of $G$; and

\item[(2)] $a_ib_j\in E(G)$ implies that $i\leq j$.
\end{itemize}
The cardinality of the largest ordered matching of $G$ is the {\it ordered matching number} of $G$ and is denoted by $\ord-match(G)$.
\end{dfn}

A graph $G$ is said to be a Cameron-Walker graph if ${\rm match}(G)=\ind-match(G)$. It is clear that a graph is Cameron-Walker if and only if all its connected components are Cameron-Walker. By \cite[Theorem 1]{cw} (see also \cite[Remark 0.1]{hhko}), a connected graph $G$ is a Cameron-Walker graph if and only if

$\bullet$ it is a star graph, or

$\bullet$ it is a star triangle, or

$\bullet$ it consists of a connected bipartite graph $H$ by vertex partition $V(H)=X\cup Y$ with the property that  there is at least one pendant edge attached  to each  vertex of $X$ and there may be some pendant triangles attached to each vertex of $Y$.

For a monomial ideal $I$, the set of its minimal primes will be denoted by ${\rm Min}(I)$. For every integer $s\geq 1$, the $s$-th {\it symbolic power} of $I$,
denoted by $I^{(s)}$, is defined to be$$I^{(s)}=\bigcap_{\frak{p}\in {\rm Min}(I)} {\rm Ker}(S\rightarrow (S/I^s)_{\frak{p}}).$$Let $I$ be a squarefree monomial ideal in $S$ and suppose that $I$ has the irredundant
primary decomposition $$I=\frak{p}_1\cap\ldots\cap\frak{p}_r,$$ where every
$\frak{p}_i$ is an ideal generated by a subset of the variables of
$S$. It follows from \cite[Proposition 1.4.4]{hh} that for every integer $s\geq 1$, $$I^{(s)}=\frak{p}_1^s\cap\ldots\cap
\frak{p}_r^s.$$We set $I^{(s)}=S$, for any integer $s\leq 0$. As we mentioned in Section \ref{sec1}, the squarefree part of $I^{(s)}$ will be denoted by $I^{\{s\}}$.

Let $I$ be a monomial ideal. The unique set of minimal monomial generators of $I$ is denoted by $G(I)$. Moreover, ${\rm Ass}(S/I)$ denotes the the set of associated primes of $S/I$. The ideal $I$ is called {\it unmixed}, if the associated primes of $S/I$ have the same height.

\begin{prop} \label{zero}
Let $I$ be a squarefree monomial ideal and $s\geq 1$ be a positive integer. Then $I^{\{s\}}\neq 0$ if and only if $s\leq \mathfrak{ht}(I)$. Moreover, if $I$ is unmixed, then $I^{\{\mathfrak{ht}(I)\}}$ is a principal ideal.
\end{prop}

\begin{proof}
Let $\mathfrak{p}=(x_{i_1}, \ldots, x_{i_k})$ be a monomial prime ideal. Then for every positive integer $s\leq k$, we have $x_1x_2 \cdots x_n\in \mathfrak{p}^s$. Also, it is clear that for each integer $s>k$, the ideal $\mathfrak{p}^s$ is not squarefree. Thus, for any monomial prime ideal $\mathfrak{p}$, the squarefree part of  $\mathfrak{p}^s$ is nonzero if and only if $s\leq \mathfrak{ht}(\mathfrak{p})$. Since $I^{\{s\}}$ is the intersection of squarefree parts of powers of finitely many prime ideals, it follows that $I^{\{s\}}\neq 0$ if and only if $s\leq \mathfrak{ht}(I)$.

To prove the last assertion, let $k$ denote the height of $I$. For any monomial prime ideal $\mathfrak{p}=(x_{i_1}, \ldots, x_{i_k})\in {\rm Ass}(S/I)$, the squarefree part of  $\mathfrak{p}^k$ is the principal ideal generated by $x_{i_1} \cdots x_{i_k}$. Since the intersection of finitely many principal monomial ideals is again a principal ideal, we conclude that $I^{\{k\}}$ is a principal ideal.
\end{proof}

We close this section by the following proposition which is an immediate consequence of restriction lemma \cite[Lemma 2.4]{hhz} and has been already observed by Erey et al. \cite{ehhs}.

\begin{prop} \label{compsq}
Let $I$ be a monomial ideal and suppose that $J$ is the squarefree part of $I$. If $J$ is a nonzero ideal, then ${\rm reg}(J)\leq {\rm reg}(I)$.
\end{prop}


\section{A Lower bound} \label{sec7}

Erey et al. \cite[Theorem 2.1]{ehhs} proved that for every graph $G$ and for each positive integer $s\leq \ind-match(G)$, the quantity $s+\ind-match(G)$ is a lower bound for the regularity of $I(G)^{[s]}$. In this short section, we prove that the same is true if one replaces $I(G)^{[s]}$ by $I(G)^{\{s\}}$ (see Theorem \ref{indmatch}). We start the proof by stating the following lemma which is an immediate consequence of \cite[Lemma 3.2]{s10}.

\begin{lem} \label{del}
Let $G$ be a graph and $x$ be a vertex of $G$. Then for every integer $s\geq 1$,$$I(G)^{\{s\}}+(x)=I(G\setminus x)^{\{s\}}+(x).$$
\end{lem}

The following proposition is the main step in the proof of Theorem \ref{indmatch}.

\begin{prop} \label{ind}
Let $G$ be a graph and assume that $H$ is an induced subgraph of $G$. Also, let $s$ be an integer such that $1\leq s\leq \mathfrak{ht}(I(H))$. Then for every pair of integers $i$ and $j$, we have$$\beta_{i,j}(I(H)^{\{s\}})\leq \beta_{i,j}(I(G)^{\{s\}}).$$In particular, ${\rm reg}(I(H)^{\{s\}})\leq {\rm reg}(I(G)^{\{s\}})$.
\end{prop}

\begin{proof}
Set $W:=V(G)\setminus V(H)$. We conclude from Lemma \ref{del} that $I(H)^{\{s\}}$ is generated by monomials in $I(G)^{\{s\}}$ which are not divisible by any variable in $W$. Then the assertions follow from restriction lemma \cite[Lemma 4.4]{hhz}.
\end{proof}

We are now ready to prove the main result of this section.

\begin{thm} \label{indmatch}
Let $G$ be a graph. Then for every integer $s$ with $1\leq s\leq \ind-match(G)$, we have$${\rm reg}(I(G)^{\{s\}})\geq s+\ind-match(G).$$
\end{thm}

\begin{proof}
Set $t:=\ind-match(G)$. Without loss of generality, suppose that the set of edges $\{x_1x_2, x_3x_4, \ldots, x_{2t-1}x_{2t}\}$ is an induced matching of $G$. Let $H$ be the induced subgraph of $G$ on vertices $x_1, x_2, \ldots, x_{2t}$. We know from the proof of \cite[Theorem 2.1]{ehhs} that ${\rm reg}(I(H)^{[s]})\geq s+t$. On the other hand, $H$ is a bipartite graph and so we deduce from \cite[Theorem 5.9]{svv} that $I(H)^{\{s\}}=I(H)^{[s]}$. Thus, ${\rm reg}(I(H)^{\{s\}})\geq s+t$ and the assertion follows from Proposition \ref{ind}.
\end{proof}


\section{Chordal graphs} \label{sec3}

The goal of this section is to provide a sharp combinatorial upper bound for the regularity of $I(G)^{\{s\}}$ when $G$ is chordal graph. More precisely, we will prove in Theorem \ref{mainchord} the for any chordal graph $G$ and for every integer $s$ with $1\leq s\leq \mathfrak{ht}(I(G))$, the regularity of $I(G)^{\{s\}}$ is bounded above by $s+\ord-match(G)$. The proof of this result is by induction and the following two lemmata have important roles in our induction step.

\begin{lem} \label{chordcolon}
Let $G$ be a graph and assume that $x_1$ is a simplicial vertex of $G$, with $N_G(x_1)=\big\{x_2, \ldots x_d\big\}$, for some integer $d\geq 2$. Then for every integer $s\geq 1$,$$\big(I(G)^{\{s\}}: x_1x_2\ldots x_d\big)=I(G\setminus N_G[x_1])^{\{s-d+1\}}.$$
\end{lem}

\begin{proof}
Let $u$ be a minimal monomial generator of $I(G\setminus N_G[x_1])^{\{s-d+1\}}$. We know from \cite[Lemma 2]{s9} that $ux_1x_2\ldots x_d\in I(G)^{(s)}$. Since $u$ is a squarefree monomial which is not divisible by $x_1, \ldots, x_d$, we deduce that $ux_1x_2\ldots x_d$ is a squarefree monomial. Hence, $ux_1x_2\ldots x_d\in I(G)^{\{s\}}$ which yields that$$I(G\setminus N_G[x_1])^{\{s-d+1\}}\subseteq \big(I(G)^{\{s\}}: x_1x_2\ldots x_d\big).$$

To prove the reverse inclusion, let $v$ be a minimal monomial generator of the ideal $\big(I(G)^{\{s\}}: x_1x_2\ldots x_d\big)$. As $vx_1x_2\ldots x_d\in I(G)^{(s)}$, we conclude from \cite[Lemma 2]{s9} that $v\in I(G)^{\{s-d+1\}}$. Since $v$ is not divisible by $x_1, \ldots, x_d$, it follows from Lemma \ref{del} that$$v\in I(G\setminus \{x_1, \ldots, x_d\})^{\{s-d+1\}}=I(G\setminus N_G[x_1])^{\{s-d+1\}}.$$
\end{proof}

In the sequel, for any subset $B$ of $\{x_1, \ldots, x_n\}$, we denote the monomial $\prod_{x_i\in B}x_i$ by $\mathrm{x}_B$.

\begin{lem} \label{sub}
Let $G$ be a graph and suppose that $W=\{x_1, \ldots ,x_d\}$ is a nonempty subset of vertices of $G$. Then for every integer $s\geq 1$,
\begin{align*}
& {\rm reg}(I(G)^{\{s\}}:x_1))\leq\\ & \max\bigg\{{\rm reg}\big(I(G\setminus A)^{\{s\}}: \mathrm{x}_B\big)+|B|-1\mid x_1\in B, A\cap B=\emptyset, A\cup B=W\bigg\}.
\end{align*}
\end{lem}

\begin{proof}
We use induction on $d$. There is nothing to prove for $d=1$. Therefore, suppose that $d\geq 2$. Set $W' =\{x_1, \ldots, x_{d-1}\}$. We know from the induction hypothesis that
\begin{align*} \tag{1} \label{3}
& {\rm reg}(I(G)^{\{s\}}:x_1))\leq\\ & \max\bigg\{{\rm reg}\big(I(G\setminus A)^{\{s\}}: \mathrm{x}_B\big)+|B|-1\mid x_1\in B, A\cap B=\emptyset, A\cup B=W'\bigg\}.
\end{align*}
For every pair of subsets $A, B\subseteq V(G)$ with $x_1\in B, A\cap B=\emptyset$ and $A\cup B=W'$, it follows from \cite[Lemma 2.10]{dhs} that
\begin{align*}
& {\rm reg}\big(I(G\setminus A)^{\{s\}}: \mathrm{x}_B\big)+|B|-1\\ & \leq \max\Big\{{\rm reg}\big((I(G\setminus A)^{\{s\}}: \mathrm{x}_B), x_d\big)+|B|-1, {\rm reg}\big(I(G\setminus A)^{\{s\}}: x_d\mathrm{x}_B\big)+|B|\Big\}\\ & =\max\Big\{{\rm reg}\big((I(G\setminus A)^{\{s\}},x_d): \mathrm{x}_B\big)+|B|-1, {\rm reg}\big(I(G\setminus A)^{\{s\}}: x_d\mathrm{x}_B\big)+|B|\Big\}\\ & =\max\Big\{{\rm reg}\big(I(G\setminus (A\cup\{x_d\}))^{\{s\}}: \mathrm{x}_B\big)+|B|-1, {\rm reg}\big(I(G\setminus A)^{\{s\}}: \mathrm{x}_{B\cup\{x_d\}}\big)+|B|\Big\},
\end{align*}
where the last equality is a consequence of Lemma \ref{del}. The claim now follows by combining the above inequality and inequality (\ref{3}).
\end{proof}

We are now ready to prove the main result of this section.

\begin{thm} \label{mainchord}
Let $G$ be a chordal graph. Then for every integer $s$ with $1\leq s\leq \mathfrak{ht}(I(G))$, we have$${\rm reg}(I(G)^{\{s\}})\leq s+\ord-match(G).$$
\end{thm}

\begin{proof}
Suppose $V(G)=\{x_1, \ldots, x_n\}$. We use induction on $n$. If $n=2$, then $I(G)=(x_1x_2)$. Hence, $\mathfrak{ht}(I(G))=1$ and so, $s=1$. In this case the assertion is obvious. Therefore, suppose $n\geq 3$.

Assume without loss of generality that $x_1$ is a simplicial vertex of $G$ and $N_G(x_1)=\big\{x_2, \ldots, x_d\big\}$, for some integer $d\geq 2$. It follows from \cite[Lemma 2.10]{dhs} that
\[
\begin{array}{rl}
{\rm reg}(I(G)^{\{s\}})\leq \max \big\{{\rm reg}(I(G)^{\{s\}}:x_1)+1, {\rm reg}(I(G)^{\{s\}},x_1)\big\}.
\end{array} \tag{2} \label{1}
\]

Using Lemma \ref{del}, we have$$\big(I(G)^{\{s\}},x_1\big)=\big(I(G\setminus x_1)^{\{s\}},x_1\big).$$Therefore, $${\rm reg}\big(I(G)^{\{s\}},x_1\big)={\rm reg}\big(I(G\setminus x_1)^{\{s\}}\big).$$As $G\setminus x_1$ is a chordal graph, it follows from the induction hypothesis and the above equality that
\[
\begin{array}{rl}
{\rm reg}\big(I(G)^{\{s\}},x_1)\leq s+\ord-match(G\setminus x_1)\leq s+\ord-match(G).
\end{array} \tag{3} \label{2}
\]

Thus, using the inequality (\ref{1}), it is enough to prove that$${\rm reg}(I(G)^{\{s\}}:x_1)\leq s+\ord-match(G)-1.$$

Substituting $W= N_G[x_1]$ in Lemma \ref{sub}, we deduce that
\begin{align*} \tag{4} \label{4}
& {\rm reg}(I(G)^{\{s\}}:x_1))\leq\\ & \max\bigg\{{\rm reg}\big(I(G\setminus A)^{\{s\}}: \mathrm{x}_B\big)+|B|-1\mid x_1\in B, A\cap B=\emptyset, A\cup B=N_G[x_1]\bigg\}.
\end{align*}

Let $A$ and $B$ be subsets of $V(G)$ with $x_1\in B$, $A\cap B=\emptyset$ and $A\cup B=N_G[x_1]$. Obviously, $G\setminus A$ is a chordal graph and $x_1$ is a simplicial vertex of this graph. Moreover, $N_{G\setminus A}[x_1]=B$. Hence, we conclude from Lemma \ref{chordcolon} that$$(I(G\setminus A)^{\{s\}}: \mathrm{x}_B)=I\big(G\setminus (A\cup B)\big)^{\{s-|B|+1\}}=I(G\setminus N_G[x_1])^{\{s-|B|+1\}}.$$Therefore, using the induction hypothesis we have
\begin{align*} \tag{5} \label{5}
& {\rm reg}\big(I(G\setminus A)^{\{s\}}: \mathrm{x}_B\big)+|B|-1\leq s-|B|+1+\ord-match(G\setminus N_G[x_1])+|B|-1\\ & \leq s+\ord-match(G)-1,
\end{align*}
where the second inequality follows from \cite[Lemma 2.1]{s4}. 

The assertion now follows by combining inequalities (\ref{1}), (\ref{2}), (\ref{4}) and (\ref{5}).
\end{proof}

\begin{rem} \label{sharpchord}
Let $G$ be a chordal graph. The bound obtained in Theorem \ref{mainchord} for the regularity of $I(G)^{\{s\}}$ is sharp. Indeed, It is clear that for any Cameron-Walker graph $G$, we have $\ind-match(G)=\ord-match(G)$. Thus, it follows from Theorem \ref{maincw} below that ${\rm reg}(I(G)^{\{s\}})=s+\ord-match(G)$  if $G$ is a Cameron-Walker graph. Moreover, there are chordal graphs $G$ for which $\ind-match(G)\neq\ord-match(G)$ and $s+\ord-match(G)$ is a sharp upper bound for ${\rm reg}(I(G)^{\{s\}})$. For instance, let $G=P_4$ be the path graph with $4$ vertices. Then $\mathfrak{ht}(I(G))=2$ and $\ord-match(G)=2$. Furthermore, one can check that ${\rm reg}(I(G)^{\{2\}})=4$. Thus, for this graph, the inequality obtained in Theorem \ref{mainchord} becomes equality for $s=2$.
\end{rem}

\begin{rem} \label{indchord}
We know from \cite{s10} that the regularity of $I(G)^{(s)}$ only depends on the induced matching number of $G$. The same is not true for the regularity of $I(G)^{\{s\}}$. For instance, let $G_1=P_4$ be the path graph with $4$ vertices and $G_2=K_3$ be the triangle graph. Then both graphs have induced matching number one. But ${\rm reg}(I(G_1)^{\{2\}})=4$ and ${\rm reg}(I(G_2)^{\{2\}})=3$.
\end{rem}


\section{Cameron-Walker graphs} \label{sec4}

Let $G$ be a Cameron-Walker graph. In \cite[Theorem 4.3]{s11} we proved that if $I(G)^{[s]}\neq 0$, then the regularity of $I(G)^{[s]}$ is equal to $s+\ind-match(G)$. In this section, we prove that the same statement is true if one replaces $I(G)^{[s]}$ by $I(G)^{\{s\}}$ (see Theorem \ref{maincw}). We know from Theorem \ref{indmatch} that for every integer $s$ with $1\leq s\leq \ind-match(G)$, the regularity of $I(G)^{\{s\}}$ is bounded below by $s+\ind-match(G)$. To prove the same lower bound for ${\rm reg}(I(G)^{\{s\}})$ when $s> \ind-match(G)$, we compute the largest degree of minimal monomial generators of $I(G)^{\{s\}}$. This will be done in Corollary \ref{gendis}. However, we need some lemmas to prove that corollary.

\begin{lem} \label{genindmatch}
Let $G$ be a graph. Then the ideal $I(G)^{\{\ind-match(G)\}}$ has a minimal monomial generator of degree $2\ind-match(G)$.
\end{lem}

\begin{proof}
We follow the notations of the proof of Theorem \ref{indmatch}. As $I(H)^{\{t\}}=I(H)^{[t]}$, we conclude that $u:=x_1x_2 \ldots x_{2t}$ is a minimal generator of $I(H)^{\{t\}}$. Therefore, Lemma \ref{del} implies that $u$ is a minimal generator of $I(G)^{\{t\}}$ too.
\end{proof}

In the next lemma, we investigate the degree of minimal monomial generators of $I(G)^{\{s\}}$ when $G$ is a star triangle graphs.

\begin{lem} \label{gentri}
Let $G$ be a star triangle graph with $t$ triangles. Then the ideal $I(G)^{\{t+1\}}$ has a minimal generator of degree $2t+1$.
\end{lem}

\begin{proof}
One can easily check that $\mathfrak{ht}(I(G))=t+1$. In particular, $I(G)^{\{t+1\}}\neq 0$. Suppose $$V(G)=\{x, y_1, y_2, \ldots, y_t, z_1, z_2, \ldots, z_t\} \ \ \ \ {\rm and} \ \ \ \ E(G)=\bigcup_{i=1}^t\{xy_i, xz_i, y_iz_i\}.$$Hence, $S = \mathbb{K}[x, y_1,\ldots,y_t, z_1, \ldots, z_t]$. Set $u:=xy_1\ldots y_tz_1\ldots z_t$. In other words, $u$ is the product of all variables in $S$. Since $I(G)^{[t+1]}$ is a nonzero squarefree monomial ideal, it contains $u$. We show that $u$ is a minimal generator of $I(G)^{[t+1]}$. As $C:=\{x, y_1, y_2, \ldots, y_t\}$ is a minimal vertex cover of $G$, we conclude that the ideal $\mathfrak{p}_C$ is a minimal prime of $I(G)$. It is obvious that $u/x\notin \mathfrak{p}_C^{t+1}$. Moreover for each integer $i=1, 2, \ldots, t$, we have $u/y_i\notin \mathfrak{p}_C^{t+1}$. Consequently, $u/x, u/y_i\notin I(G)^{\{t+1\}}$. By symmetry, $u/z_i\notin I(G)^{\{t+1\}}$. Hence, $u$ is a minimal generator of $I(G)^{\{t+1\}}$ and this proves the lemma.
\end{proof}

In the following lemma, we study the degree of minimal monomial generators of $I(G)^{\{s\}}$ when $G$ is a connected Cameron-Walker graph.

\begin{lem} \label{gencon}
Suppose $G$ is a connected Cameron-Walker graph with $\ind-match(G)< \mathfrak{ht}(I(G))$ and let $s$ be an integer such that $\ind-match(G)+1\leq s\leq \mathfrak{ht}(I(G))$. Then the ideal $I(G)^{\{s\}}$ has a minimal monomial generator of degree $s+\ind-match(G)$.
\end{lem}

\begin{proof}
If $G$ is a bipartite graph, then it follows from K${\rm\ddot{o}}$nig's theorem \cite[Theorem 3.1.16]{w} that$$\ind-match(G)={\rm match}(G)=\mathfrak{ht}(I(G)).$$So, there is nothing to prove in this case. Hence, assume that $G$ is not a bipartite graph. Suppose that $G$ is a star triangle graph and denote the number of triangles of $G$ by $t$. One can easily check that $\ind-match(G)=t$ and $\mathfrak{ht}(I(G))=t+1$. Therefore, $s=t+1$ and the assertion follows from Lemma \ref{gentri}. So, assume that $G$ consists of a connected bipartite graph $H$ with vertex partition $V(H)=X\cup Y$  such that there is at least one pendant edge attached to each vertex of $X$ and that there may be some pendant triangles attached to each vertex of $Y$.

As above, let $t$ denote the number of triangles of $G$. Since $G$ is not a bipartite graph, we have $t\geq 1$. Suppose $X=\{x_1, \ldots, x_p\}$ and $Y=\{y_1, \ldots, y_q\}$. Also, assume that for each integer $i=1, 2, \ldots, p$, the edges $x_iw_{i1}, \ldots, x_iw_{im_i}$ are the pendant edges attached to $x_i$, and for each integer $j=1, 2, \ldots, q$, the triangles with vertices $\{y_j, z_{j1}, z'_{j1}\}, \ldots, \{y_j, z_{jr_j}, z'_{jr_j}\}$ are the pendant triangles attached to $y_j$. In particular, $r_1+r_2+\cdots +r_q=t$. We know from \cite[Lemma 3.3]{s8} that $\ind-match(G)=p+t$. Let $k$ be the number of integers $j$ with $1\leq j\leq  q$ such that $r_j\neq 0$. We claim that $\mathfrak{ht}(I(G))=p+k+t$.

Note that the set$$\{x_1, \ldots, x_p\}\cup\{y_j, z_{j1}, \ldots, z_{jr_j}\mid 1\leq j\leq q \ {\rm and} \ r_j\neq 0\}$$ is a vertex cover of $G$ with cardinality $p+k+t$. In particular, $\mathfrak{ht}(I(G))\leq p+k+t$. The reverse inequality follows from the following observations.

{\bf (I)} Let $C$ be a vertex cover of $G$. Then for each pair of integers $i$ and $\ell$ with $1\leq i\leq p$ and $1\leq \ell\leq m_i$, we have $C\cap \{x_i, w_{\ell}\}\neq \emptyset$. Since $m_1, \ldots, m_p$ are nonzero integers, we deduce that$$|C\cap\{x_1, \ldots, x_p, w_{11}, \ldots, w_{1m_1}, w_{21}, \ldots w_{2m_2}, \ldots, w_{p1}, \ldots, w_{pm_p}\}|\geq p.$$

{\bf (II)} Let $C$ be a vertex cover of $G$. Then for each pair of integers $j$ and $\ell$ with $1\leq j\leq q$ and $1\leq \ell\leq r_j$, we have $|C\cap \{y_j, z_{j\ell}, z'_{j\ell}\}|\geq 2$. Therefore, $$|C\cap\{y_1, \ldots, y_q, z_{11}, \ldots, z_{1r_1}, \ldots, z_{q1}, \ldots, z_{qr_q}, z'_{11}, \ldots, z'_{1r_1}, \ldots, z'_{q1}, \ldots, z'_{qr_q}\}|\geq k+t.$$

The above observations imply that $\mathfrak{ht}(I(G))\geq p+k+t$ and the claim follows.

As $\ind-match(G)+1\leq s\leq \mathfrak{ht}(I(G))$, we may write $s=p+t+m$ for some integer $m$ with $1\leq m\leq k$. Without loss of generality, suppose that $r_1, \ldots, r_k\geq 1$ and $r_{k+1}= \cdots =r_q=0$. In other words, there is at least one pendant triangle attached to each of the vertices $y_1, \ldots, y_k$ and there is no pendant triangle attached to $y_{k+1}, \ldots, y_q$. Set$$u:=\prod_{i=1}^px_i\prod_{i=1}^pw_{i1}\prod_{j=1}^my_j\prod_{j=1}^k\prod_{\ell=1}^{r_j}(z_{j\ell}z'_{j\ell}).$$In particular, $u$ is a squarefree monomial of degree $2p+m+2t=s+\ind-match(G)$. We prove that $u$ is a minimal generator of $I(G)^{\{s\}}$. By observations (I) and (II) above, for every vertex cover $C$ of $G$, we have $u\in \mathfrak{p}_C^{p+t+m}=\mathfrak{p}_C^s$. Consequently, $u\in I(G)^{\{s\}}$. Set$$C_1:=\{x_1, \ldots, x_p, y_1, \ldots, y_k, z_{11}, \ldots, z_{1r_1}, z_{21}, \ldots, z_{2r_2},\ldots, z_{k1}, \ldots, z_{kr_k}\}.$$Then $C_1$ is a vertex cover of $G$. One can easily check that the monomials$$u/x_1, \ldots, u/x_p, u/y_1, \ldots, u/y_m, u/z_{11}, \ldots, u/z_{1r_1}, \ldots, u/z_{k1}, \ldots, u/z_{kr_k}$$do not belong to $\mathfrak{p}_{C_1}^s$. In particular, the above monomials do not belong to $I(G)^{\{s\}}$. Set$$C_2:=\{w_{11}, \ldots, w_{1m_1}, \ldots, w_{p1}, \ldots, w_{pm_p}, y_1, \ldots, y_q, z'_{11}, \ldots, z'_{1r_1}, \ldots, z'_{k1}, \ldots, z'_{kr_k}\}.$$Then $C_2$ is a vertex cover of $G$. One can easily check that the monomials$$u/w_{11}, \ldots, u/w_{p1}, u/z'_{11}, \ldots, u/z'_{1r_1}, \ldots, u/z'_{k1}, \ldots, u/z'_{kr_k}$$do not belong to $\mathfrak{p}_{C_2}^s$. Hence, these monomials do not belong to $I(G)^{\{s\}}$. Consequently, $u$ is a minimal generator of $I(G)^{\{s\}}$.
\end{proof}

In the following corollary, we extend the assertion of Lemma \ref{gencon} to disconnected Cameron-Walker graphs.

\begin{cor} \label{gendis}
Suppose $G$ is a (not necessarily connected) Cameron-Walker graph and let $s$ be an integer with $\ind-match(G)\leq s\leq \mathfrak{ht}(I(G))$. Then the ideal $I(G)^{\{s\}}$ has a minimal generator of degree $s+\ind-match(G)$. In particular,$${\rm reg}(I(G)^{\{s\}})\geq s+\ind-match(G).$$
\end{cor}

\begin{proof}
Let $G_1, \ldots, G_c$ denote the connected components of $G$. Note that$$\ind-match(G)=\sum_{i=1}^c\ind-match(G_i) \ \ \ \ {\rm and} \ \ \ \ \mathfrak{ht}(I(G))=\sum_{i=1}^c\mathfrak{ht}(I(G)).$$Since $\ind-match(G)\leq s\leq \mathfrak{ht}(I(G))$, we may write $s=s_1+\cdots +s_c$ such that $\ind-match(G_i)\leq s_i\leq \mathfrak{ht}(I(G_i))$, for every integer $i$ with $1\leq i\leq c$. Using Lemmata \ref{genindmatch} and \ref{gencon}, each ideal $I(G_i)^{\{s_i\}}$ has a minimal monomial generator of degree $s_i+\ind-match(G_i)$. Then $u:=u_1u_2\cdots u_c$ is a minimal generator of $I(G)^{\{s\}}$ with ${\rm deg}(u)=s+\ind-match(G)$. The second assertion immediately follows from the first one.
\end{proof}

We are now ready to prove the main result of this section.

\begin{thm} \label{maincw}
Let $G$ be a Cameron-Walker graph. Then for every integer $s$ with $1\leq s\leq \mathfrak{ht}(I(G))$, we have$${\rm reg}(I(G)^{\{s\}})=s+\ind-match(G).$$
\end{thm}

\begin{proof}
By Theorem \ref{indmatch} and Corollary \ref{gendis}, for every integer $s$ with $1\leq s\leq \mathfrak{ht}(I(G))$, we have$${\rm reg}(I(G)^{\{s\}})\geq s+\ind-match(G).$$So, it is enough to show that$${\rm reg}(I(G)^{\{s\}})\leq s+\ind-match(G).$$

Without loss of generality, assume that $G$ has no isolated vertex and suppose $V(G)=\{x_1, \ldots, x_n\}$. We use induction on $n$. If $n=2$, then $I(G)=(x_1x_2)$. Hence, $\mathfrak{ht}(I(G))=1$ and so, $s=1$. In this case the assertion is obvious. Therefore, suppose $n\geq 3$. If $G$ is a bipartite graph, then by \cite[Theorem 5.9]{svv}, we have $I(G)^{\{s\}}=I(G)^{[s]}$ and the assertion follows from \cite[Theorem 4.3]{s11} (note that by K${\rm\ddot{o}}$nig's theorem \cite[Theorem 3.1.16]{w}, we have $\ind-match(G)={\rm match}(G)=\mathfrak{ht}(I(G))$). Therefore, assume that $G$ is not a bipartite graph. It follows from the structure of Cameron-Walker graphs that $G$ has a pendant triangle, say $T$. Suppose without loss of generality that $V(T)=\{x_1, x_2, x_3\}$ and that ${\rm deg}_G(x_2)={\rm deg}_G(x_3)=2$. Consider the following short exact sequence.
\begin{align*}
0 \longrightarrow \frac{S}{(I(G)^{\{s\}}:x_1)}(-1)\longrightarrow \frac{S}{I(G)^{\{s\}}}\longrightarrow \frac{S}{I(G)^{\{s\}}+(x_1)}\longrightarrow 0
\end{align*}
It yields that
\[
\begin{array}{rl}
{\rm reg}(I(G)^{\{s\}})\leq \max\big\{{\rm reg}(I(G)^{\{s\}}:x_1)+1, {\rm reg}(I(G)^{\{s\}},x_1)\big\}.
\end{array} \tag{6} \label{11}
\]
As $G\setminus x_1$ is a (disconnected) Cameron-Walker graph with$$\ind-match(G\setminus x_1)\leq\ind-match(G),$$we conclude from Lemma \ref{del} and the induction hypothesis that
\[
\begin{array}{rl}
&{\rm reg}(I(G)^{\{s\}},x_1)={\rm reg}(I(G\setminus x_1)^{\{s\}},x_1)={\rm reg}(I(G\setminus x_1)^{\{s\}})\\ &\leq s+\ind-match(G\setminus x_1)
\leq s+\ind-match(G).
\end{array} \tag{7} \label{12}
\]

Therefore, using inequality (\ref{11}), we need to show that$${\rm reg}(I(G)^{\{s\}}:x_1)\leq s+\ind-match(G)-1.$$

Consider the following short exact sequence.
\begin{align*}
0 \longrightarrow \frac{S}{(I(G)^{\{s\}}:x_1x_2)}(-1)\longrightarrow \frac{S}{(I(G)^{\{s\}}:x_1)}\longrightarrow \frac{S}{(I(G)^{\{s\}}:x_1)+(x_2)}\longrightarrow 0
\end{align*}
It follows that
\[
\begin{array}{rl}
{\rm reg}(I(G)^{\{s\}}:x_1)\leq \max\big\{{\rm reg}(I(G)^{\{s\}}:x_1x_2)+1, {\rm reg}\big((I(G)^{\{s\}}:x_1), x_2\big)\big\}.
\end{array} \tag{8} \label{13}
\]

\vspace{0.3cm}
{\bf Claim 1.} ${\rm reg}(I(G)^{\{s\}}:x_1x_2)\leq s+\ind-match(G)-2$.

\vspace{0.3cm}
{\it Proof of Claim 1.} Consider the following short exact sequence.
\begin{align*}
0 \longrightarrow \frac{S}{(I(G)^{\{s\}}:x_1x_2x_3)}(-1)\longrightarrow \frac{S}{(I(G)^{\{s\}}:x_1x_2)}\longrightarrow \frac{S}{(I(G)^{\{s\}}:x_1x_2)+(x_3)}\longrightarrow 0
\end{align*}
It yields that
\[
\begin{array}{rl}
{\rm reg}(I(G)^{\{s\}}:x_1x_2)\leq \max\big\{{\rm reg}(I(G)^{\{s\}}:x_1x_2x_3)+1, {\rm reg}\big((I(G)^{\{s\}}:x_1x_2), x_3\big)\big\}.
\end{array} \tag{9} \label{14}
\]

It follows from \cite[Lemma 3.3]{s8} that $\ind-match(G\setminus N_G[x_2])=\ind-match(G)-1$. Moreover, $x_2$ is a simplicial vertex of $G$ with $N_G[x_2]=\{x_1, x_2, x_3\}$.  Therefore, we conclude from Lemma \ref{chordcolon} and the induction hypothesis that
\[
\begin{array}{rl}
&{\rm reg}(I(G)^{\{s\}}:x_1x_2x_3)={\rm reg}(I(G\setminus N_G[x_2])^{\{s-2\}})\leq (s-2)+\ind-match(G\setminus N_G[x_2])\\ &=s+\ind-match(G)-3.
\end{array} \tag{10} \label{15}
\]
On the other hand,$${\rm reg}\big((I(G)^{\{s\}}:x_1x_2), x_3\big)={\rm reg}\big((I(G)^{\{s\}}, x_3):x_1x_2\big)={\rm reg}\big(I(G\setminus x_3)^{\{s\}}:x_1x_2\big),$$where the second equality is a consequence of Lemma \ref{del}. As $x_2$ is a leaf of $G\setminus x_3$, it follows from Lemma \ref{chordcolon} that$$\big(I(G\setminus x_3)^{\{s\}}:x_1x_2\big)=I(G\setminus \{x_1, x_2, x_3\})^{\{s-1\}}.$$Therefore,$${\rm reg}\big((I(G)^{\{s\}}:x_1x_2), x_3\big)={\rm reg}\big(I(G\setminus \{x_1, x_2, x_3\})^{\{s-1\}}\big).$$Note that $G\setminus \{x_1, x_2, x_3\}$ is a Cameron-Walker graph. Moreover, we know from \cite[Lemma 3.3]{s8} that $\ind-match(G\setminus \{x_1, x_2, x_3\})=\ind-match(G)-1$. Thus, we deduce from the above equality and the induction hypothesis that
\[
\begin{array}{rl}
&{\rm reg}\big((I(G)^{\{s\}}:x_1x_2), x_3\big)\leq (s-1)+\ind-match(G\setminus\{x_1, x_2, x_3\})\\ &=s+\ind-match(G)-2.$$
\end{array} \tag{11} \label{16}
\]
Finally, the assertion of Claim 1 follows from inequalities (\ref{14}), (\ref{15}) and (\ref{16}).

\vspace{0.3cm}
{\bf Claim 2.} ${\rm reg}\big((I(G)^{\{s\}}:x_1), x_2\big)\leq s+\ind-match(G)-1$.

\vspace{0.3cm}
{\it Proof of Claim 2.} Consider the following short exact sequence.
\begin{align*}
0 \longrightarrow &\frac{S}{\big(\big((I(G)^{\{s\}}:x_1),x_2\big): x_3\big)}(-1)\longrightarrow \frac{S}{\big((I(G)^{\{s\}}:x_1),x_2\big)}\longrightarrow\\ &\frac{S}{\big((I(G)^{\{s\}}:x_1),x_2, x_3\big)}\longrightarrow 0
\end{align*}
It follows that
\[
\begin{array}{rl}
&{\rm reg}\big((I(G)^{\{s\}}:x_1), x_2\big)\leq\\ &\max\big\{{\rm reg}\big(\big((I(G)^{\{s\}}:x_1),x_2\big): x_3\big)+1, {\rm reg}\big((I(G)^{\{s\}}:x_1),x_2, x_3\big)\big\}.
\end{array} \tag{12} \label{17}
\]
Notice that
\begin{align*}
& {\rm reg}\big(\big((I(G)^{\{s\}}:x_1),x_2\big): x_3\big)={\rm reg}\big((I(G)^{\{s\}},x_2): x_1x_3\big)={\rm reg}\big((I(G\setminus x_2)^{\{s\}},x_2): x_1x_3\big)\\ &={\rm reg}\big(\big(I(G\setminus x_2)^{\{s\}}: x_1x_3\big),x_2\big)={\rm reg}\big(I(G\setminus x_2)^{\{s\}}: x_1x_3\big),
\end{align*}
where the second equality is a consequence of Lemma \ref{del}. As $x_1x_3$ is a pendant edge of $G\setminus x_2$, it follows from Lemma \ref{chordcolon} that$$\big(I(G\setminus x_2)^{\{s\}}:x_1x_3\big)=I(G\setminus \{x_1, x_2, x_3\})^{\{s-1\}}.$$Therefore,
\[
\begin{array}{rl}
& {\rm reg}\big(\big((I\{G\}^{\{s\}}:x_1),x_2\big): x_3\big)={\rm reg}\big(I(G\setminus \{x_1, x_2, x_3\})^{\{s-1\}}\big)\\ & \leq s+\ind-match(G)-2,
\end{array} \tag{13} \label{18}
\]
where the inequality is know by the argument of the proof of Claim 1.
On the other hand,
\begin{align*}
& {\rm reg}\big((I(G)^{\{s\}}:x_1),x_2, x_3\big)={\rm reg}\big((I(G)^{\{s\}},x_2,x_3): x_1\big)\\ &={\rm reg}\big((I(G\setminus \{x_2,x_3\})^{\{s\}},x_2, x_3): x_1\big)\\ &={\rm reg}\big(\big(I(G\setminus \{x_2,x_3\})^{\{s\}}: x_1\big),x_2, x_3\big)\\ &={\rm reg}\big(I(G\setminus \{x_2,x_3\})^{\{s\}}: x_1\big),
\end{align*}
where the second equality follows from Lemma \ref{del}.

We know from \cite[Lemma 4.2]{s3} that$${\rm reg}\big(I(G\setminus \{x_2,x_3\})^{\{s\}}: x_1\big)\leq {\rm reg}\big(I(G\setminus \{x_2,x_3\})^{\{s\}}\big).$$Consequently,$${\rm reg}\big((I(G)^{\{s\}}:x_1),x_2, x_3\big)\leq {\rm reg}\big(I(G\setminus \{x_2,x_3\})^{\{s\}}\big).$$Notice that $G\setminus \{x_2, x_3\}$ is a Cameron-Walker graph. Moreover, it follows from \cite[Lemma 3.3]{s8} that $\ind-match(G\setminus \{x_2, x_3\})=\ind-match(G)-1$. So, we deduce from the induction hypothesis and the above inequality that
\[
\begin{array}{rl}
& {\rm reg}\big((I(G)^{\{s\}}:x_1),x_2, x_3\big)\leq s+\ind-match(G\setminus \{x_2, x_3\})\\ & =s+\ind-match(G)-1.
\end{array} \tag{14} \label{19}
\]

It now follows from inequalities (\ref{17}), (\ref{18}) and (\ref{19}) that$${\rm reg}\big((I(G)^{\{s\}}:x_1), x_2\big)\leq s+\ind-match(G)-1,$$and this proves Claim 2.

\vspace{0.3cm}
We deduce from Claims 1, 2, and inequality (\ref{13}) that$${\rm reg}(I(G)^{\{s\}}:x_1)\leq s+\ind-match(G)-1.$$The above inequality together with inequalities (\ref{11}) and $\ref{12}$ implies that$${\rm reg}(I(G)^{\{s\}})\leq s+\ind-match(G)$$and this completes the proof of the theorem.
\end{proof}

\begin{rem} \label{orsynoeq}
By \cite[Theorem 4.3]{s11}, for any Cameron-Walker graph $G$ and for every integer $s$ with $1\leq s\leq {\rm match}(G)$, we have ${\rm reg}(I(G)^{[s]})=s+\ind-match(G)$. Thus, it follows from Theorem \ref{maincw} that for this class of graphs, the equality ${\rm reg}(I(G)^{\{s\}})={\rm reg}(I(G)^{[s]})$ holds, for each integer $s$ with $1\leq s\leq {\rm match}(G)$. It is natural to ask whether the same is true for any arbitrary graph $G$. However, the answer is negative. For instance, let $G$ be the graph with edge ideal $I(G)=(x_1x_2, x_1x_3, x_1x_4, x_2x_3)$. Then one can easily check that $I(G)^{[2]}=(x_1x_2x_3x_4)$ and $I(G)^{\{2\}}=(x_1x_2x_3)$. Hence, ${\rm reg}(I(G)^{[2]})=4$ while ${\rm reg}(I(G)^{\{2\}})=3$.
\end{rem}


\section{Second power} \label{sec5}

In this section, we study the regularity of the squarefree part of second symbolic power of edge ideals. It is known by \cite[Theorem 1.1]{bn} and \cite[Corollary 3.9]{s12} that for any graph $G$, $${\rm reg}(I(G)^2)\leq {\rm reg}(I(G))+2 \ \ \ {\rm and} \ \ \ {\rm reg}(I(G)^{(2)})\leq {\rm reg}(I(G))+2.$$Tuus, we conclude the following corollary from Proposition \ref{compsq}.

\begin{cor} \label{twth}
Let $G$ be a graph.
\begin{itemize}
\item[(i)] If $\rm match(G)\geq 2$, then ${\rm reg}(I(G)^{[2]})\leq {\rm reg}(I(G))+2$.
\item[(ii)] If $\mathfrak{ht}(I(G))\geq 2$, then ${\rm reg}(I(G)^{\{2\}})\leq {\rm reg}(I(G))+2$.
\end{itemize}
\end{cor}

In Theorem \ref{sym2}, we will prove that the regularity of $I(G)^{\{2\}}$ is bounded above by ${\rm match}(G)+2$. The proof of that theorem is based on the inequality provided in Lemma \ref{rfirst} for $s=1$. Lemma \ref{rfirst} will be used also in Section \ref{sec6} to bound the regularity of $I(G)^{\{3\}}$. Because of this reason, we state and prove it for a general $s$ (and not only for $s=1$). In order to prove Lemma \ref{rfirst}, we use the following proposition whose proof is essentially the same as that of \cite[Proposition 3.1]{s11}. However, we include the proof for the sake of completeness.

\begin{prop} \label{matchgenord}
Assume that $G$ is a graph and set $s:={\rm match}(G)$. Then the monomials in $G(I(G)^{[s]})$ can be labeled as $u_1, \ldots, u_m$ such that
for every pair of integers $1\leq j< i\leq m$, there exists an integer $r\leq i-1$ such that $(u_r:u_i)$ is generated by a variable, and $(u_j:u_i)\subseteq (u_r:u_i)$.
\end{prop}

\begin{proof}
Using \cite[Theorem 4.12]{b}, the elements of $G(I(G)^s)$ can be labeled as $v_1, \ldots, v_t$ such that for every pair of integers $1\leq j< i\leq t$, one of the following conditions holds.
\begin{itemize}
\item [(1)] $(v_j:v_i) \subseteq (I(G)^{s+1}:v_i)$; or
\item [(2)] there exists an integer $k\leq i-1$ such that $(v_k:v_i)$ is generated by a variable, and $(v_j:v_i)\subseteq (v_k:v_i)$.
\end{itemize}

Since $G(I(G)^{[s]})\subseteq G(I(G)^s)$, there exist integers $\ell_1, \ldots, \ell_m$ such that $G(I(G)^{[s]})=\{v_{\ell_1}, \ldots, v_{\ell_m}\}$. For every integer $k$ with $1\leq k\leq m$, set $u_k:=v_{\ell_k}$. We claim that this labeling satisfies the desired property. To prove the claim, we fix integers $i$ and $j$ with $1\leq j< i\leq m$. Based on properties (1) and (2) above, we divide the rest of the proof into two cases.

\vspace{0.3cm}
{\bf Case 1.} Assume that $(v_{\ell_j}:v_{\ell_i}) \subseteq (I(G)^{s+1}:v_{\ell_i})$. Recall that $v_{\ell_i}$ and $v_{\ell_j}$ are squarefree monomials. Therefore, $(v_{\ell_j}:v_{\ell_i})=(u)$, for some squarefree monomial $u$ with ${\rm gcd}(u, v_{\ell_i})=1$. This yields that $uv_{\ell_i}$ is a squarefree monomial in $I(G)^{s+1}$ which is a contradiction, as $I(G)^{[s+1]}=0$ by the definition of $s$.

\vspace{0.3cm}
{\bf Case 2.} Assume that there exists an integer $k\leq \ell_i-1$ such that $(v_k:v_{\ell_i})$ is generated by a variable, and $(v_{\ell_j}:v_{\ell_i})\subseteq (v_k:v_{\ell_i})$. Hence, $(v_k:v_{\ell_i})=(x_p)$, for some integer $p$ with $1\leq p\leq n$. It follows from the inclusion  $(v_{\ell_j}:v_{\ell_i})\subseteq (v_k:v_{\ell_i})$ that $x_p$ divides $v_{\ell_j}/{\rm gcd}(v_{\ell_j}, v_{\ell_i})$. Since, $v_{\ell_j}$ is a squarefree monomial, we deuce that $x_p$ does not divide $v_{\ell_i}$. As ${\rm deg}(v_k)={\rm deg}(v_{\ell_i})$, it follows from $(v_k:v_{\ell_i})=(x_p)$ that there is a variable $x_q$ dividing $v_{\ell_i}$ such that $v_k=x_pv_{\ell_i}/x_q$. This implies that $v_k$ is a squarefree monomial. Hence, $v_k=v_{\ell_r}=u_r$, for some integer $r$ with $1\leq r\leq m$. Using $k\leq \ell_i-1$, we have $\ell_r\leq \ell_i-1$. Therefore, $r\leq i-1$ and$$(u_j:u_i)\subseteq(u_r:u_i)=(x_p).$$ This completes the proof.
\end{proof}

In the following lemma, we provide an a method to bound the regularity of squarefree part symbolic powers of edge ideals.

\begin{lem} \label{rfirst}
Assume that $G$ is a graph and $s$ is an integer with $1\leq s\leq {\rm match}(G)$. Let $G(I(G)^{[s]})=\{u_1, \ldots, u_m\}$ denote the set of minimal monomial generators of $I(G)^{[s]}$. Then$${\rm reg}(I(G)^{\{s+1\}})\leq \max\bigg\{{\rm reg}\big(I(G)^{\{s+1\}}:u_i\big)+2s, 1\leq i\leq m, {\rm reg}\big(I(G)^{\{s+1\}}+I(G)^{[s]}\big)\bigg\}.$$
\end{lem}

\begin{proof}
By Proposition \ref{matchgenord} and \cite[Proposition 3.1]{s11}, we may assume that for every pair of integers $1\leq j< i\leq m$, one of the following conditions holds.
\begin{itemize}
\item [(i)] $(u_j:u_i) \subseteq (I(G)^{[s+1]}:u_i)\subseteq (I(G)^{\{s+1\}}:u_i)$; or
\item [(ii)] there exists an integer $k\leq i-1$ such that $(u_k:u_i)$ is generated by a variable, and $(u_j:u_i)\subseteq (u_k:u_i)$.
\end{itemize}
We conclude from (i) and (ii) above that
\[
\begin{array}{rl}
\big((I(G)^{\{s+1\})}+ (u_1, \ldots, u_{i-1})):u_i\big)=(I(G)^{\{s+1\}}:u_i)+({\rm some \ variables}).
\end{array}
\]
Hence, it follows from \cite[Lemma 2.10]{b} that
\[
\begin{array}{rl}
{\rm reg}\big((I(G)^{\{s+1\}}, u_1, \ldots, u_{i-1}):u_i\big) \leq {\rm reg}(I(G)^{\{s+1\}}:u_i).
\end{array} \tag{15} \label{20}
\]

For every integer $i$ with $0\leq i\leq m$, set $I_i:=(I(G)^{\{s+1\}}, u_1, \ldots, u_i)$. In particular, $I_0=I(G)^{\{s+1\}}$ and $I_m=I(G)^{\{s+1\}}+I(G)^{[s]}$. Consider the short exact sequence
$$0\rightarrow S/(I_{i-1}:u_i)(-2s)\rightarrow S/I_{i-1}\rightarrow S/I_i\rightarrow 0,$$
for every $1\leq i\leq m$. It follows that$${\rm reg}(I_{i-1})\leq \max \big\{{\rm reg}(I_{i-1}:u_i)+2s, {\rm reg}(I_i)\big\}.$$Therefore,
\begin{align*}
& {\rm reg}(I(G)^{\{s+1\}})={\rm reg}(I_0)\leq \max\big\{{\rm reg}(I_{i-1}:u_i)+2s, 1\leq i\leq m, {\rm reg}(I_m)\big\}\\ & =\max\big\{{\rm reg}(I_{i-1}:u_i)+2s, 1\leq i\leq m, {\rm reg}(I(G)^{\{s+1\}}+I(G)^{[s]})\big\}.
\end{align*}
The assertion now follows from the inequality (\ref{20}).
\end{proof}

According to Lemma \ref{rfirst}, for bounding the regularity of $I(G)^{\{2\}}$, one needs to study the ideals of the form $(I(G)^{\{2\}}:u)$ where $u$ is a minimal monomial generator of $I(G)$. The identity provided in the following lemma would be useful for this study.

\begin{lem} \label{intsec}
Let $G$ be a graph and suppose that $x_ix_j$ is an edge of $G$. Then for any integer $s$ with $1\leq s\leq\mathfrak{ht}(I(G))-1$, we have$$\big(I(G)^{\{s+1\}}:x_ix_j\big)=\big(I(G-x_j)^{\{s\}}:x_i\big)\cap \big(I(G-x_i)^{\{s\}}:x_j\big).$$
\end{lem}

\begin{proof}
Let $u$ be a minimal monomial generator of $\big(I(G)^{\{s+1\}}:x_ix_j\big)$. Then $ux_ix_j$ is a squarefree monomial in $I(G)^{(s+1)}$. Consequently, $u$ is not divisible by $x_i$ and $x_j$. We conclude from $ux_ix_j\in I(G)^{(s+1)}$ that $ux_i\in I(G)^{(s)}$. As $ux_i$ is a squarefree monomial and does not divide $x_j$, we deduce from Lemma \ref{del} that $ux_i\in I(G-x_j)^{\{s\}}$. Therefore, $u\in \big(I(G-x_j)^{\{s\}}:x_i\big)$. Similarly, $u$ belongs to $\big(I(G-x_i)^{\{s\}}:x_j\big)$. Hence,$$\big(I(G)^{\{s+1\}}:x_ix_j\big)\subseteq\big(I(G-x_j)^{\{s\}}:x_i\big)\cap \big(I(G-x_i)^{\{s\}}:x_j\big).$$

To prove the reverse inclusion, let $v$ be a minimal monomial generator of$$\big(I(G-x_j)^{\{s\}}:x_i\big)\cap \big(I(G-x_i)^{\{s\}}:x_j\big).$$In particular, $v$ is not divisible by $x_i$ and $x_j$. So, $vx_ix_j$ is a squarefree monomial. Hence, we must show that $vx_ix_j\in I(G)^{(s+1)}$. It is enough to prove that for any minimal vertex cover $C$ of $G$, we have $vx_ix_j\in \mathfrak{p}_C^{s+1}$. Let $C$ be a minimal vertex cover of $G$. It follows from $x_ix_j\in E(G)$ that $C$ contains at least one of the vertices $x_i$ and $x_j$. Without lose of generality, suppose $x_i\in C$. Since $v\in \big(I(G-x_i)^{\{s\}}:x_j\big)$, we have$$vx_j\in I(G-x_i)^{(s)}\subseteq I(G)^{(s)}\subseteq \mathfrak{p}_C^s.$$ This together with $x_i\in \mathfrak{p}_C$ implies that $vx_ix_j\in \mathfrak{p}_C^{s+1}$.
\end{proof}

We are now ready to prove the main result of this section.

\begin{thm} \label{sym2}
For any graph $G$, with $\mathfrak{ht}(I(G))\geq 2$ we have$${\rm reg}(I(G)^{\{2\}})\leq \min\big\{{\rm reg}(I(G))+2, {\rm match}(G)+2\big\}.$$In particular, Conjecture \ref{conjsym} is true for $s=2$.
\end{thm}

\begin{proof}
We know from Corollary \ref{twth} that ${\rm reg}(I(G)^{\{2\}})\leq {\rm reg}(I(G))+2$. So, we should prove that$${\rm reg}(I(G)^{\{2\}})\leq {\rm match}(G)+2.$$

Let $G(I(G))=\{u_1, \ldots, u_m\}$ denote the set of minimal monomial generators of $I(G)$. Notice that $I(G)^{\{2\}}+I(G)=I(G)$. Therefore,$${\rm reg}(I(G)^{\{2\}}+I(G))={\rm reg}(I(G))\leq {\rm match}(G)+1,$$where the inequality follows from \cite[Theorem 6.7]{hv}. Hence, using Lemma \ref{rfirst}, it is enough to prove that for every integer $i$ with $1\leq i\leq m$,$${\rm reg}(I(G)^{\{2\}}:u_i)\leq {\rm match}(G).$$We prove the above inequality in the following lemma.
\end{proof}

\begin{lem} \label{colonsym2}
Let $G$ be a graph with $\mathfrak{ht}(I(G))\geq 2$ and suppose $x_ix_j$ is an edge of $G$. Then$${\rm reg}(I(G)^{\{2\}}:x_ix_j)\leq {\rm match}(G).$$
\end{lem}

\begin{proof}
We divide the proof into three cases.

\vspace{0,3cm}
{\bf Case 1.} Assume that $x_ix_j$ is a pendant edge of $G$. Then it follows from Lemma \ref{chordcolon} that $(I(G)^{\{2\}}:x_ix_j)=I(G\setminus\{x_i,x_j\})$. Note that ${\rm match}(G\setminus\{x_i,x_j\})\leq {\rm match}(G)-1$. Thus, using \cite[Theorem 6.7]{hv}, we have$${\rm reg}(I(G)^{\{2\}}:x_ix_j)={\rm reg}(I(G\setminus\{x_i,x_j\}))\leq {\rm match}(G\setminus\{x_i,x_j\})+1\leq {\rm match}(G).$$

\vspace{0,3cm}
{\bf Case 2.} Assume that $x_i$ and $x_j$ are vertices of degree two in a pendant triangle of $G$. In other words, ${\rm deg}_G(x_i)={\rm deg}_G(x_j)=2$ and there is a vertex $x_k\in V(G)$ such that $x_ix_k, x_jx_k\in E(G)$. Then clearly, we have $x_k\in (I(G)^{\{2\}}:x_ix_j)$. Hence,
\begin{align*}
& (I(G)^{\{2\}}:x_ix_j)=(I(G)^{\{2\}}:x_ix_j)+(x_k)=\big((I(G)^{\{2\}}, x_k):x_ix_j\big)\\ & =\big((I(G-x_k)^{\{2\}}, x_k):x_ix_j\big)=(I(G-x_k)^{\{2\}}:x_ix_j)+(x_k),
\end{align*}
where the third equality is a consequence of Lemma \ref{del}. Therefore,$${\rm reg}\big(I(G)^{\{2\}}:x_ix_j\big)={\rm reg}\big(I(G-x_k)^{\{2\}}:x_ix_j\big).$$As the vertices $x_i, x_j, x_k$ form a pendant triangle of $G$, we deduce that the edge $x_ix_j$ is a connected component of $G-x_k$. Thus, using Lemma \ref{chordcolon} and the above equality, we have$${\rm reg}\big(I(G)^{\{2\}}:x_ix_j\big)={\rm reg}\big(I(G\setminus\{x_i, x_j, x_k\})\big).$$Since ${\rm match}(G\setminus\{x_i, x_j, x_k\})\leq {\rm match}(G)-1$, we deduce from \cite[Theorem 6.7]{hv} that$${\rm reg}\big(I(G\setminus\{x_i, x_j, x_k\})\big)\leq {\rm match}(G\setminus\{x_i, x_j, x_k\})+1\leq {\rm match}(G).$$So, the assertion follows in this case.

\vspace{0,3cm}
{\bf Case 3.} Assume that the degrees of $x_i$ and $x_j$ are at least two and moreover, $|N_G(x_i)\cup N_G(x_j)|\geq 4$. We know from Lemma \ref{intsec} that$$\big(I(G)^{\{2\}}:x_ix_j\big)=\big(I(G-x_j):x_i\big)\cap \big(I(G-x_i):x_j\big).$$Set$$I_1:=(I(G-x_j):x_i) \ \ \ {\rm and} \ \ \ I_2:=(I(G-x_i):x_j).$$Therefore, $\big(I(G)^{\{2\}}:x_ix_j\big)=I_1\cap I_2$. Consequently, we have the following short exact sequence.
\begin{align*}
0\rightarrow \frac{S}{(I(G)^{\{2\}}:x_ix_j)}\rightarrow \frac{S}{I_1}\oplus \frac{S}{I_2}\rightarrow \frac{S}{I_1+I_2}\rightarrow 0.
\end{align*}
Applying \cite[Corollary 18.7]{p'} on the above exact sequence yields that
\[
\begin{array}{rl}
{\rm reg}\big(I(G)^{\{2\}}:x_ix_j\big)\leq \max\{{\rm reg}(I_1), {\rm reg}(I_2), {\rm reg}(I_1+I_2)+1\}.
\end{array} \tag{16} \label{21}
\]
Note that$$I_1=(I(G-x_j):x_i)=I(G-N_G[x_i])+({\rm some \ variables}).$$Hence, by \cite[Lemma 2.10]{b}, we have ${\rm reg}(I_1)\leq {\rm reg}(I(G-N_G[x_i]))$. Observe that $G-N_G[x_i]$ is an induced subgraph of $G-\{x_i, x_j\}$. Then a similar argument as in Case 1 implies that ${\rm reg}(I(G-N_G[x_i]))\leq {\rm match}(G)$ and so,
\[
\begin{array}{rl}
{\rm reg}(I_1)\leq {\rm match}(G).
\end{array} \tag{17} \label{22}
\]
By symmetry,
\[
\begin{array}{rl}
{\rm reg}(I_2)\leq {\rm match}(G).
\end{array} \tag{18} \label{23}
\]

Moreover, notice that$$I_1+I_2=(I(G-x_j):x_i)+(I(G-x_i):x_j)=I\big(G-(N_G[x_i]\cup N_G[x_j])\big)+({\rm some \ variables}).$$Using \cite[Lemma 2.10]{b}, we deduce that
\[
\begin{array}{rl}
{\rm reg}(I_1+I_2)\leq {\rm reg}\big(I(G-(N_G[x_i]\cup N_G[x_j]))\big).
\end{array} \tag{19} \label{24}
\]
As we are assuming that the degrees of $x_i$ and $x_j$ are at least two and $|N_G(x_i)\cup N_G(x_j)|\geq 4$, there are distinct vertices $x_p, x_q\notin\{x_i, x_j\}$ such that $x_ix_p$ and $x_jx_q$ are edges of $G$. Then for every matching $M$ of $G-(N_G[x_i]\cup N_G[x_j])$, the set $M\cup\{x_ix_p, x_jx_q\}$ is a matching of $G$. In particular,$${\rm match}\big(G-(N_G[x_i]\cup N_G[x_j])\big)+2\leq {\rm match}(G).$$Thus, we conclude from \cite[Theorem 6.7]{hv} that
\[
\begin{array}{rl}
{\rm reg}\big(I(G-(N_G[x_i]\cup N_G[x_j]))\big)+1\leq {\rm match}\big(G-(N_G[x_i]\cup N_G[x_j])\big)+2\leq {\rm match}(G).
\end{array} \tag{20} \label{26}
\]
Inequalities (\ref{24}) and (\ref{26}) imply that
\[
\begin{array}{rl}
{\rm reg}(I_1+I_2)+1\leq {\rm match}(G).
\end{array} \tag{21} \label{27}
\]
The assertion now follows by combining inequalities (\ref{21}), (\ref{22}), (\ref{23}) and (\ref{27}).
\end{proof}


\section{Third power} \label{sec6}

In this section, we investigate the regularity of the squarefree part of $I(G)^{\{3\}}$. As the main result, in Theorem \ref{mainpo3}, we provide a sharp upper bound for the regularity of $I(G)^{\{3\}}$. The proof of Theorem \ref{mainpo3} is based on the inequality obtained in Lemma \ref{rfirst}. So, we need to analyze the ideals of the form $(I(G)^{\{3\}}:u)$ where $u$ is a minimal monomial generator of $I(G)^{[2]}$. In Lemma \ref{colsy}, we will see that these ideals can be expressed in terms of squarefree part of second symbolic powers of suitable ideals. In order to prove Lemma \ref{colsy}, we need the following lemma.

\begin{lem} \label{symor}
Let $G$ be a graph and suppose that $e=x_1x_2$ is an edge of $G$. Then$$\big(I(G)^{(2)}:x_1x_2\big)+(x_1, x_2)=\big(I(G)^{\{2\}}:x_1x_2\big)+(x_1, x_2).$$
\end{lem}

\begin{proof}
The inclusion "$\supseteq$" is obvious. To prove the reverse inclusion, let $u$ be a minimal monomial generator of $(I(G)^{(2)}:x_1x_2)+(x_1, x_2)$. We may assume that $u$ is not divisible by $x_1$ and $x_2$, as otherwise there is nothing to prove. Therefore, $u$ is a minimal monomial generator of $(I(G)^{(2)}:x_1x_2)$. In particular, $ux_1x_2\in I(G)^{(2)}$. By \cite[Lemma 3.2]{s12}, we know that $u$ is a squarefree monomial. As $u$ is not divisible by $x_1$ and $x_2$, it follows that $ux_1x_2$ is a squarefree monomial as well. Consequently, $u\in (I(G)^{\{2\}}:x_1x_2)$.
\end{proof}

\begin{lem} \label{colsy}
Assume that $G$ is a graph with edge set $E(G)=\{e_1, \ldots, e_r\}$, and let $s$ be an integer with $1\leq s\leq {\rm match}(G)$. Suppose  $u=e_{i_1}\ldots e_{i_s}$ is a minimal monomial generator of $I(G)^{[s]}$. Then$$\big(I(G)^{\{s+1\}}:u\big)=\bigg(\big(I(G)^{\{2\}}:e_{i_1}\big)^{\{s\}}: e_{i_2}\ldots e_{i_s}\bigg).$$
\end{lem}

\begin{proof}
There is nothing to prove for $s=1$. So suppose that $s\geq 2$.

Let $v$ be a minimal monomial generator of $\bigg(\big(I(G)^{\{2\}}:e_{i_1}\big)^{\{s\}}: e_{i_2}\ldots e_{i_s}\bigg)$. Then $v$ is a squarefree monomial which is not divisible by the variables in $\cup_{j=1}^se_{i_j}$. It follows from \cite[Lemma 3.3]{s12} that $v$ belongs to $\big(I(G)^{(s+1)}:u\big)$. Hence, $vu\in I(G)^{(s+1)}$. As $uv$ is a squarefree monomial, we conclude that $uv\in I(G)^{\{s+1\}}$. In other words $v\in (I(G)^{\{s+1\}}:u)$.

Next, we prove the reverse inclusion. Without loss of generality, we may suppose that $e_{i_1}=x_1x_2$. By \cite[Lemma 3.2]{s12}, we know that $(I(G)^{(2)}:e_{i_1})$ is a squarefree monomial ideal. Let $(I(G)^{(2)}:e_{i_1})=\bigcap_{k=1}^m\mathfrak{p}_k$ be the irredundant primary decomposition of $(I(G)^{(2)}:e_{i_1})$. In particular, the ideals $\mathfrak{p}_1, \ldots, \mathfrak{p}_m$ are generated by subsets of variables. Then
\begin{align*}
& \big(I(G)^{(2)}:e_{i_1}\big)^{\{s\}}+(x_1, x_2)=\bigcap_{k=1}^m\mathfrak{p}_k^{\{s\}}+(x_1, x_2) =\bigcap_{k=1}^m(\mathfrak{p}_k+(x_1, x_2))^{\{s\}}+(x_1, x_2)\\ & =J^{\{s\}}+(x_1, x_2),
\end{align*}
where $J=(I(G)^{(2)}:e_{i_1})+(x_1, x_2)$. On the other hand, it follows from Lemma \ref{symor} that $J=(I(G)^{\{2\}}:e_{i_1})+(x_1, x_2)$. Therefore, we conclude from the above equalities that
\[
\begin{array}{rl}
& \big(I(G)^{(2)}:e_{i_1}\big)^{\{s\}}+(x_1, x_2)=\big((I(G)^{\{2\}}:e_{i_1})+(x_1, x_2)\big)^{\{s\}}+(x_1, x_2)=\\ & \big((I(G)^{\{2\}}:e_{i_1})\big)^{\{s\}}+(x_1, x_2).
\end{array} \tag{22} \label{28}
\]

Let $w$ be a minimal monomial generator of the ideal $(I(G)^{\{s+1\}}:u)$. Then $w$ is a squarefree monomial which is not divisible by the variables in $\cup_{j=1}^se_{i_j}$. We deduce from \cite[Lemma 3.3]{s12} that $w$ belongs to the ideal $\bigg(\big(I(G)^{(2)}:e_{i_1}\big)^{(s)}: e_{i_2}\ldots e_{i_s}\bigg)$. As $we_{i_2}\ldots e_{i_s}$ is a squarefree monomial, we conclude that$$we_{i_2}\ldots e_{i_s}\in \big(I(G)^{(2)}:e_{i_1}\big)^{\{s\}}.$$Since $we_{i_2}\ldots e_{i_s}$ is not divisible by $x_1$ and $x_2$, we deduce from (\ref{28}) that$$we_{i_2}\ldots e_{i_s}\in \big(I(G)^{\{2\}}:e_{i_1}\big)^{\{s\}}.$$This yields that$$w\in \bigg(\big(I(G)^{\{2\}}:e_{i_1}\big)^{\{s\}}: e_{i_2}\ldots e_{i_s}\bigg)$$and we are done.
\end{proof}

Let $G$ be a graph and assume that $u=e_1e_2$ is a minimal monomial generator of $I(G)^{[2]}$. We know from Lemma \ref{colsy} that $(I(G)^{\{3\}}:u)=\big((I(G)^{\{2\}}:e_1)^{\{2\}}: e_2\big)$. Therefore, to bound the regularity of $(I(G)^{\{3\}}:u)$, one needs to study the ideals of the form $(I(G)^{\{2\}}:e)$ where $e$ is an edge of $G$. In the following lemma, we determine the structure of this type of ideals.

\begin{lem} \label{seccoldesc}
Let $G$ be a graph and let $e=x_ix_j$ be an edge of $G$. Then
\begin{align*}
& \big(I(G)^{\{2\}}:x_ix_j\big)=I(G-\{x_i, x_j\})+\big(x_px_q: x_p\in N_{G-x_j}(x_i), x_q\in N_{G-x_i}(x_j), x_p\neq x_q\big)\\ & +\big(x_t: x_t\in N_G(x_i)\cap N_G(x_j)\big).
\end{align*}
\end{lem}

\begin{proof}
Using Lemma \ref{symor} and \cite[Lemma 3.2]{s12}, we have
\begin{align*}
& \big(I(G)^{\{2\}}:x_ix_j\big)+(x_i, x_j)=\big(I(G)^{(2)}:x_ix_j\big)+(x_i, x_j)\\ & =I(G)+\big(x_px_q: x_p\in N_G(x_i), x_q\in N_G(x_j), x_p\neq x_q\big)\\ & +\big(x_t: x_t\in N_G(x_i)\cap N_G(x_j)\big)+(x_i, x_j)\\ & =I(G-\{x_i, x_j\})+\big(x_px_q: x_p\in N_{G-x_j}(x_i), x_q\in N_{G-x_i}(x_j), x_p\neq x_q\big)\\ & +\big(x_t: x_t\in N_G(x_i)\cap N_G(x_j)\big)+(x_i, x_j).
\end{align*}
Since the minimal monomial generators of $(I(G)^{\{2\}}:x_ix_j)$ are not divisible by $x_i$ and $x_j$, the assertion follows from the above equalities.
\end{proof}

In the following lemma, we determine a combinatorial upper bound for the regularity of $(I(G)^{\{3\}}:u)$ when $u$ is a minimal monomial generator of $I(G)^{[2]}$.

\begin{lem} \label{lemreg}
Assume that $G$ is a graph with edge set $E(G)=\{e_1, \ldots, e_r\}$ and ${\rm match}(G)\geq 2$. Suppose $u=e_{i_1}e_{i_2}$ is a minimal monomial generator of $I(G)^{[2]}$. If $I(G)^{\{3\}}\neq 0$, then$${\rm reg}\big(I(G)^{\{3\}}:u\big)\leq \frac{|V(G)|}{2}-1.$$
\end{lem}

\begin{proof}
Without loss of generality, we may assume that $e_{i_1}=x_1x_2$ and $e_{i_2}=x_3x_4$. Let $G'$ be the graph with edge ideal$$I(G')=I(G-\{x_1, x_2\})+\big(x_px_q: x_p\in N_{G-x_2}(x_1), x_q\in N_{G-x_1}(x_2), x_p\neq x_q\big).$$In particular, $|V(G')\leq |V(G)|-2$. Using Lemma \ref{seccoldesc}, there exists a subset $A$ of variables with the property that$$\big(I(G)^{\{2\}}:e_{i_1}\big)=I(G')+(A).$$We know from Lemma \ref{colsy} that
\begin{align*} \tag{23} \label{29}
& {\rm reg}\big(I(G)^{\{3\}}:u\big)={\rm reg}\big(\big(I(G)^{\{2\}}:e_{i_1}\big)^{\{2\}}: e_{i_2}\big) ={\rm reg}\big(\big(I(G'), A\big)^{\{2\}}: e_{i_2}\big)\\ & ={\rm reg}\big(\big(I(G'\setminus A), A\big)^{\{2\}}: e_{i_2}\big),
\end{align*}
Where the last equality is a consequence of Lemma \ref{del}. We divide the rest of the proof into two cases.

\vspace{0.3cm}
{\bf Case 1.} Suppose $A\cap \{x_3, x_4\}=\emptyset$. Therefore, $e_{i_2}$ is an edge of $G\setminus (A\cup\{x_1, x_2\})$. Thus, $e_{i_2}$ belongs to $I(G'\setminus A)$. Hence, for every variable $z\in A$, we have$$ze_{i_2}\in \big(I(G'\setminus A), A\big)^{\{2\}}.$$In other words,$$z\in \big(\big(I(G'\setminus A), A\big)^{\{2\}}: e_{i_2}\big).$$Therefore,$$\big(\big(I(G'\setminus A)\big)^{\{2\}}: e_{i_2}\big)+(A)\subseteq \big(\big(I(G'\setminus A), A\big)^{\{2\}}: e_{i_2}\big).$$

Conversely, let $v$ be a monomial in $\big(\big(I(G'\setminus A), A\big)^{\{2\}}: e_{i_2}\big)$ and suppose that $v$ is not divisible by any variable in $A$. As $e_{i_2}$ has no common vertex with $A$, we deduce that $ve_{i_2}\in I(G'\setminus A)^{\{2\}}$. Consequently,
\[
\begin{array}{rl}
\big(\big(I(G'\setminus A), A\big)^{\{2\}}: e_{i_2}\big)=\big(\big(I(G'\setminus A)\big)^{\{2\}}: e_{i_2}\big)+(A).
\end{array} \tag{24} \label{30}
\]
If $I(G'\setminus A)\big)^{\{2\}}$ is the zero ideal, then we conclude from the above equality that $\big(\big(I(G'\setminus A), A\big)^{\{2\}}: e_{i_2}\big)=(A)$ which is generated by a (possibly empty) subset of variable. So, we conclude from (\ref{29}) an the above equality that$${\rm reg}\big(I(G)^{\{3\}}:u\big)={\rm reg}(A)\leq 1\leq \frac{|V(G)|}{2}-1.$$So, suppose that $I(G'\setminus A)\big)^{\{2\}}\neq 0$. In particular, $\mathfrak{ht}(I(G'\setminus A)\geq 2$. It follows from equalities (\ref{29}) and (\ref{30}) that$${\rm reg}\big(I(G)^{\{3\}}:u\big)={\rm reg}\big(\big(I(G'\setminus A)\big)^{\{2\}}: e_{i_2}\big)\leq {\rm match}(G'\setminus A)\leq \frac{|V(G')|}{2}\leq \frac{|V(G)|}{2}-1,$$where the first inequality is known by Lemma \ref{colonsym2} and the last inequality follows from the fact that $|V(G')|\leq |V(G)|-2$.

\vspace{0.3cm}
{\bf Case 2.} Assume that $A\cap \{x_3, x_4\}\neq\emptyset$. In particular, $|A|\geq 1$. Without loss of generality, suppose $x_3\in A$ and set $B:=A\setminus \{x_3\}$. Then by (\ref{29}) we have
\begin{align*} \tag{25} \label{31}
& {\rm reg}\big(I(G)^{\{3\}}:u\big)={\rm reg}\big(\big(I(G'\setminus A), A\big)^{\{2\}}: e_{i_2}\big)={\rm reg}\big(\big(\big(I(G'\setminus A), A\big)^{\{2\}}: x_3\big):x_4\big)\\ & ={\rm reg}\big(\big(I(G'\setminus A), B\big):x_4\big),
\end{align*}
where the last equality follows from the assumption that $x_3\in A$. If $x_4\in B$, then $\big(I(G'\setminus A), B\big):x_4\big)=S$ and the assertion would be trivial. So, assume that $x_4\notin B$. It follows that$$\big(I(G'\setminus A), B\big):x_4\big)=I\big(G'\setminus (A\cup N_{G-A}[x_4])\big)+({\rm some \ vriables}).$$In particular,
\[
\begin{array}{rl}
{\rm reg}\big(\big(I(G'\setminus A), B\big):x_4\big)={\rm reg}\big(I\big(G'\setminus (A\cup N_{G-A}[x_4])\big)\big).
\end{array} \tag{26} \label{32}
\]
Note that $x_1, x_2, x_3$ and $x_4$ are not vertices of $G'\setminus (A\cup N_{G-A}[x_4])$. Therefore, using equalities (\ref{31}), (\ref{32}) and \cite[Theorem 6.7]{hv}, we have
\begin{align*}
& {\rm reg}\big(I(G)^{\{3\}}:u\big)\leq {\rm match}\big(G'\setminus (A\cup N_{G-A}[x_4])\big)+1\leq \frac{|V(G'\setminus (A\cup N_{G-A}[x_4]))|}{2}+1\\ & \leq\frac{|V(G)|-4}{2}+1=\frac{|V(G)|}{2}-1.
\end{align*}
\end{proof}

According to Lemma \ref{rfirst}, for bounding the regularity of $I(G)^{\{3\}}$, one needs to estimate the regularity of $I(G)^{\{3\}}+I(G)^{[2]}$. The following lemma shows that this ideal is nothing other that $I(G)^{[2]}$.

\begin{lem} \label{ttsym}
For any graph $G$, we have $I(G)^{\{3\}}\subseteq I(G)^{[2]}$.
\end{lem}

\begin{proof}
The claim is an immediate consequence of \cite[Proposition 3.7]{s12}.
\end{proof}

We are now ready to prove the main result of this section.

\begin{thm} \label{mainpo3}
Let $G$ be a graph with $n$ vertices such that $\mathfrak{ht}(I(G))\geq 3$. Then$${\rm reg}(I(G)^{\{3\}})\leq\min\Big\{\Big\lfloor\frac{n}{2}\Big\rfloor+3, {\rm reg}(I(G))+4\Big\}.$$
\end{thm}

\begin{proof}
We know from \cite[Corollary 3.9]{s12} that ${\rm reg}(I(G)^{(3)})\leq {\rm reg}(I(G))+4$. Hence, we deduce from Proposition \ref{compsq} that$${\rm reg}(I(G)^{\{3\}})\leq{\rm reg}(I(G))+4.$$

In order to prove the inequality ${\rm reg}(I(G)^{(3)})\leq \lfloor n/2\rfloor+3$, note that by Lemma \ref{ttsym}, we have $I(G)^{\{3\}}+I(G)^{[2]}=I(G)^{[2]}$. Thus, we conclude from Lemmata \ref{rfirst} and \ref{lemreg} that$${\rm reg}(I(G)^{\{3\}})\leq\min\big\{\frac{n}{2}+3, {\rm reg}(I(G)^{[2]})\big\}.$$The assertion now follows from \cite[Theorem 2.11]{ehhs} which states that$${\rm reg}(I(G)^{[2]})\leq {\rm match}(G)+2\leq \frac{n}{2}+2.$$
\end{proof}

Let $G$ be a graph with $n$ vertices such that $I(G)^{\{3\}}\neq 0$. In Theorem \ref{mainpo3}, we proved that $\big\lfloor\frac{n}{2}\big\rfloor+3$ and ${\rm reg}(I(G))+4$ are upper bounds for the regularity of $I(G)^{\{3\}}$. The following examples show that both bounds are sharp.

\begin{exmps} \label{sharp3}
\begin{itemize}
\item[(1)] Let $G=C_5$ be the $5$-cycle graph. Then one may easily check that $I(G)^{\{3\}}=(x_1x_2x_3x_4x_5)$. Therefore,$${\rm reg}(I(G)^{\{3\}})=5=\bigg\lfloor\frac{|V(G)|}{2}\bigg\rfloor+3$$which is strictly smaller that ${\rm reg}(I(G))+4$.
\item[(2)] Let $G$ be the complete bipartite graph $K_{3,5}$. Since the complementary graph $\overline{G}$ is chordal, it follows from \cite[Theorem 1]{f}  that ${\rm reg}(I(G))=2$. Moreover, as $G$ is a bipartite graph, we deduce from \cite[Theorem 5.9]{svv} that $I(G)^{\{3\}}=I(G)^{[3]}$. Consequently, $I(G)^{\{3\}}$ is generated in degree $6$. Thus using Theorem \ref{mainpo3}, we have$${\rm reg}(I(G)^{\{3\}})=6={\rm reg}(I(G))+4<\bigg\lfloor\frac{|V(G)|}{2}\bigg\rfloor+3.$$
\end{itemize}
\end{exmps}





\end{document}